\documentclass[12pt]{amsart}

\usepackage{amssymb, amsthm} 
\newtheorem{theorem}{Theorem}

\newtheorem{claim}[theorem]{Claim}

\newtheorem{corollary}[theorem]{Corollary}
\newtheorem{definition}[theorem]{Definition}

\newtheorem{lemma}[theorem]{Lemma}

\theoremstyle{remark}
\newtheorem{remark}[theorem]{Remark}
\newcommand{\R}{\mathbf{R}}
\newcommand{\pd}{ {\partial}}
\newcommand{\lin}{\mathcal{L}}
\newcommand{\el}{\mathcal{E}}
\newcommand{\eps}{\varepsilon}
\newcommand{\Uplus}{U_1+\eps^2 v_a}

\newcommand{\OmRN}{\Omega_{R, N}}
\newcommand{\rrt}{_{R, R_1, \tau_1}}

\begin{document}

\title[Construction of  Self-Similar Surfaces under MCF II]{Construction of Complete Embedded Self-Similar Surfaces under Mean Curvature Flow. Part II. }
\author{Xuan Hien Nguyen}
\address{Department of Mathematical Sciences, University of Cincinnati, Cincinnati,
OH, 45221} 
\email{hien.nguyen@uc.edu}

\begin{abstract} We study the Dirichlet problem associated to the equation for self-similar surfaces for graphs over the Euclidean plane with a disk removed. We show the existence of a solution provided the boundary conditions on the boundary circle  are small enough and satisfy some symmetries.  This is the second step towards the construction of new examples of complete embedded self similar surfaces under mean curvature flow.
\end{abstract}

\maketitle

\section{Introduction}

This paper is the second one of a series of three articles describing the construction of new examples of complete embedded self-similar surfaces under mean curvature flow \cite{mine;part1}\cite{mine;part3}.  Our general strategy is inspired by  Kapouleas'  article  \cite{kapouleas;embedded-minimal-surfaces} and is described in detail in the previous installment. 
Let us recall it briefly:  to construct a new self-similar surface, we take two known examples suitably positioned and replace a neighborhood of their intersection with an appropriately bent scaled Scherk's singly periodic surface.  The procedure is called desingularization. 

The resulting surface is not smooth; however it is a good approximate solution. The next task is to find  a small function whose graph over it satisfies the self-similar surface equation
	\begin{equation}
	\label{eq:self-shrinker}
	H+X\cdot \nu =0,
	\end{equation}
where $H$ is the mean curvature and $\nu$ is the normal vector so that the mean curvature vector is $\mathbf{H}=H \nu$. The sign of $H$ is chosen so that the mean curvature of a convex surface is positive. 	Before considering graphs of functions on the entire surface, we have to work locally and study the Dirichlet problems with small boundary conditions  corresponding to \eqref{eq:self-shrinker}  for graphs of functions on the different pieces. 

In the first part \cite{mine;part1}, we found that a perturbation of an appropriately scaled and bent Scherk surface satisfies \eqref{eq:self-shrinker} and can be used to desingularize the intersection of a cylinder and a plane or a sphere and a plane. In this article, we study the Dirichlet problem  corresponding to \eqref{eq:self-shrinker}  for  functions on the outer plane, which is the Euclidean plane with a disk removed. We show that the Dirichlet problem possesses a solution if the boundary conditions on the circle are small enough. Our methods here are different from the one used in Part I. Since the equation \eqref{eq:self-shrinker} can be written with a single set of coordinates on the outer plane, our tools come mainly from  partial differential equation theory, with some interior estimates established by Ecker and Huisken \cite{ecker-huisken;interior-estimates} for the Mean Curvature Flow. The third article is to discuss the gluing of the solutions on the different pieces in a manner to obtain a smooth complete embedded self-similar surface.

\subsection{Main result} 
Let $2> R> \sqrt 3 /2$ and $N\geq5$, consider a function $f \in C^4([0, 2\pi])$ satisfying the symmetries
	\begin{equation}
	\label{eq:symmetries-f}
	f(\theta) = -f(-\theta) = f(\pi/N - \theta).
	\end{equation}

	\begin{theorem}
	\label{thm:main-1}
There is an $\eps_0 >0$ depending on $R$ and $N$ such that for any $f: [0, 2 \pi] \to \R$ with $\| f \|_{C^4[0, 2 \pi]}=\eps \leq \eps_0$ and $f$ satisfying the symmetries above, there is a function $u$ on $\Omega = \R^2\setminus \bar B_R$ such that 
	\begin{gather}
	\textrm{ the graph of $u$ over $\Omega$ satisfies $H + X\cdot \nu =0$} \label{eq1}\\
	u =f \textrm{ on } \pd B_R, \label{eq:boundary-condition}\\
	u(r, \theta)=-u(r, -\theta) = u(r, \pi/N -\theta) \text{ for } r >R, \theta \in [0, 2\pi)\label{eq:symmetries-u}
	\end{gather}
In addition, We can choose the constant $\eps_0$ uniformly for all $R$ such that  $2>R>\sqrt 3/2$.
	\end{theorem}
Here,  $u=f$ on $\pd B_R$ means that $u(R, \theta) = f(\theta)$ in polar coordinates. This is a slight abuse of notation which does not induce confusion. Note that the conditions \eqref{eq:symmetries-u} imply that $u$ is odd with respect to rotations of $180$ degrees with respect to the $x$-axis and even with respect to reflections across the planes $\theta = \pi/2N+ k \pi/N$, which are exactly the symmetries imposed in Part I \cite{mine;part1}.

Computing the mean curvature and the normal vector $\nu$ of the graph of $u$ using coordinates in $\R^2$, we find that \eqref{eq1} is equivalent to 
	\begin{equation}
	\label{eq:quasilinear}
	\el(u)=g^{ij}(Du(\xi))D_{ij}u(\xi) - \xi \cdot D u(\xi) + u(\xi) = 0, \quad \xi \in \Omega,
	\end{equation}
where $g^{ij}(Du) = \delta^{ij} - \frac{D_i u D_j u }{1+|Du|^2}$. This is our main equation. Note that our domain is unbounded and that the coefficient in front of $u$ is positive, so the standard maximum principle is not applicable directly. 

We first study the  properties of the linear operator $\lin$ associated to $\el$. From solutions to the linear problem, we then construct a sub and supersolution to the quasilinear equation $\el u =0$.

The solution to the elliptic equation \eqref{eq:quasilinear} is found as a limit for time $\tau$ going to infinity of a solution $u$ to the parabolic equation $\pd_{\tau} u =\el u$. For initial conditions with bounded gradient, standard theory on parabolic equations assures the existence of a solution to the initial value problem on a short time interval $[0, \tau_0)$. The constructed sub and supersolution serve as barriers for solutions of the parabolic equation $\pd_{\tau} u =\el u$ and allow us to control the growth of $u$ at infinity in the space variable for all time. Exploiting the close relation  between the parabolic equation $\pd_{\tau} u =\el u$ and the mean curvature flow, we use interior estimates from the latter to bound derivatives of our solution to the former away from the boundary. Loosely speaking, this gives us control of the derivatives outside of  a bounded annulus around $\pd \Omega$. 
To bridge the gap, we invoke standard parabolic theory on bounded sets and obtain uniform estimates for derivatives of $u$ for all $\tau \in [0, \tau_0)$. This implies that the solution has to exists for all time $\tau \in [0, \infty)$. We conclude that a subsequence of $u(\cdot, \tau)$  tends to a solution of $\el u=0$ as $\tau$ goes to infinity by proving a monotonicity formula. 

Some sections  of this paper appeared in the author's thesis, written under the direction of Sigurd Angenent at the University of Wisconsin-Madison. The author is indebted to him for invaluable discussion.

%
\section{Definitions}
\label{sec:definitions}

\subsection{Mean curvature flow}
The properties of  solutions to \eqref{eq1} are closely related to properties of solutions to the mean curvature flow (MCF), which will be used intensively. Let us therefore define it.

Let $X(\cdot, t): M^2 \to \R^3$ be a one parameter family of immersions of $2$-dimensional smooth hypersurfaces in $\R^3$. We say that $M_t = X(M^2, t)$  is a solution to the mean curvature flow if
	\begin{gather}
	\frac{d}{dt} X(p, t) = \mathbf{H}(p, t), \quad p \in M, t>0 \label{eq:mcf1}\\
	X(p, 0) = X_0 (p)\nonumber
	\end{gather}
is satisfied for some initial data $X_0$. Here $\mathbf{H}(p, t)$ is the mean curvature vector of the hypersurface $M_t$ at $X(p, t)$. 

If the hypersurfaces $M_t$ can be written locally as graphs of a function $v(\cdot, t)$ over a domain in the $xy$- plane, the quasilinear equation
	\begin{equation}
	\label{eq:mcf}
	\frac{d}{dt} v = g^{ij}(Dv) D_{ij}v, \quad g^{ij}(Dv) = \delta_{ij}-\frac{D_{i}v D_{j}v}{1+|D v|^2}
	\end{equation}
is equivalent to \eqref{eq:mcf1} up to tangential diffeomorphisms.

\subsection{Weighted Hilbert Sobolev spaces $H$, $V$ and $V_0$}

We consider the Hilbert space
\[
H= L^2(\Omega, dm(\xi)),
\]
where $\Omega$ is an open subset of $\R^n$ (not necessarily bounded)
and $m$ is the Gaussian measure
$
dm(\xi) = e^{-\frac12|\xi|^2}d\xi.
$
We write $\|f\|_H$ and $(f, g)_H$ for the norm and inner product of $H$. We also define the Hilbert space
\[
V= \{ f\in H : \pd_if\in H \text{ for }i=1,\ldots,n\}.
\]
The inner product of $V$ is
$
(f,g)_V = \int_{\Omega}\{fg+D f\cdot D g\} dm(\xi).
$
Let $V_0$ be the closure of $C^\infty_c(\Omega)$ in $V$. Even when $\Omega$ is unbounded the inclusion $V_0\subset H$ is
compact. 

\section{The Linear Operator}
\label{sec:linear-operator}

Let us define sections $\Omega_{R, N}$ of the outer plane by
	\begin{equation}
	\Omega_{R, N} = \{ (r, \theta) : R<r , - \pi/N < \theta < \pi/N \} \label{eq:OmRN}
	\end{equation}
and their corresponding Hilbert spaces
	\begin{align*}
	V_{0; Sym}(\Omega_{R,N})  &= \{ u \in V_0(\Omega_{R, N}) \mid u(r, \theta) = -u(r, -\theta) = u(r, \pi/N - \theta)\} \\
	H_{Sym}(\Omega_{R, N}) &= \{ u \in H(\Omega_{R, N}) \mid u(r, \theta) = -u(r, -\theta) = u(r, \pi/N - \theta)\} 
	\end{align*}
In this section, we study the properties of $\lin$, the linear operator associated to $\el$ defined in \eqref{eq:quasilinear},
 	\[
	\lin (u) = \Delta u- \xi \cdot D u+ u.
	\]


\subsection{Eigenvalues and eigenfunctions}
\label{ssec:eigenvalues}
Consider the operator 
	\[
	\lin_-(u)=\Delta u- \xi \cdot D u- u
	\]
on $\Omega_{0, N}$. The compactness of the inclusion $V_0(\Omega_{0, N}) \subset H(\Omega_{0, N})$ yields the following theorem:

\begin{theorem}
\label{thm:countable-eigenvalues}
The operators $\lin_-$ and $\lin$ have a countable set of eigenvalues having no limit point except possibly $\lambda = \infty$ and a corresponding $H$-orthogonal basis of eigenfunctions $\varphi_i \in V_0(\Omega_{0, N})$. 
\end{theorem}

\begin{proof} Theory on compact bounded self-adjoint operators in Hilbert spaces gives us the result for $\lin_-$. The result for $\lin$ then  follows immediately.
\end{proof}



Consider eigenfunctions of the form $f(r) g(\theta)$ satisfying the symmetries
	\begin{equation}
	\label{eq:sym}
	g(\theta) = -g(-\theta) = g(\pi/N - \theta).
	\end{equation}
There is an eigenvalue $\lambda \in \R$ such that 
	\[
	\lin( f(r) g(\theta)) = f_{rr} g+ \frac{1}{r} f_r g + \frac{1}{r^2} f g_{\theta\theta} - r f_r g +fg = \lambda fg. 
	\]
Separation of variables and the symmetry conditions  \eqref{eq:sym} on $g$ imply that the only possibilities for $g$ are $g= C \sin (m\theta)$ with $m=kN$, $ k \in \mathbf{N}$. Therefore
	\begin{gather*}
	g_{\theta\theta} =- m^2 g\\
	\lin_0(f) =  f_{rr}+\frac{1}{r}f_r-r f_r+(1-\lambda) f - \frac{1}{r^2}m^2 f =0.
	\end{gather*}
We have
	\[
	\lin_0 (r^n) = (n^2 -m^2) r^{n-2} + (1 -\lambda - n) r^n.
	\]
$\lin_0$ is a linear map on the vector space $V=\{ r^m F(r^2) \mid F \textrm{ polynomial } \} $ and it is upper diagonal in the basis $\{ r^m, r^{m+2}, r^{m+4}, \ldots \}$. The operator $\lin_0$ has a zero eigenvalue if one of the entry on the diagonal vanishes, i.e. if $1-\lambda-(m+2l) =0$. 	We therefore found some of the eigenvalues of our original operator $\lin$,
	\begin{equation}
	\label{eq:eigenvalues}
	\lambda = 1-(kN+2l), \quad k \in \mathbf{N}, l \in \{ 0\} \cup \mathbf N,
	\end{equation}
where $\mathbf N =\{ 1, 2, \ldots \}$ is the set of natural numbers. For each value of $k, l$ the corresponding eigenfunction is of the form 
	\begin{equation}
	\label{eq:eigenfunctions}
	r^{kN} P_{k, l}(r^2) \sin(k N \theta)
	\end{equation}
where $P_{k, l}$ is a polynomial of degree $l$. \\


We now show that the eigenvalues values given by  \eqref{eq:eigenvalues} are in fact all the possible eigenvalues of $\lin$.  This is done by proving the following theorem:

\begin{theorem}
\label{thm:completeness}
The set $S$ of eigenfunctions \eqref{eq:eigenfunctions} corresponding to the eigenvalues \eqref{eq:eigenvalues} forms a basis for  the Hilbert space $H_{Sym}(\Omega_{0, N})$.
\end{theorem}

Before we start the proof,  note that for each fixed $k \in \mathbf N$ and $s \neq l$,
	\[
	(r^{kN} P_{k, l} \sin(kN\theta) , r^{kN} P_{k, s} \sin(kN\theta))_{H(\Omega_{0, N})} =0
	\]
since they are eigenfunctions corresponding to different eigenvalues. Hence,
	$
	\int_{r=0}^{\infty}  P_{k, l} (r^2) P_{k, s}(r^2) r^{2kN} e^{-r^2/2} r \ dr =0.
	$
By the change of variable $\rho = r^2$,
	\begin{equation}
	\int_{\rho =0}^{\infty} P_{k, l}(\rho) P_{k, s}(\rho) \rho^{kN} e^{-\rho/2} d\rho =0 \label{P-orthogonal}.
	\end{equation}
The polynomials $P_{k, l}$, which are closely related to the general Laguerre polynomials, are therefore orthogonal in $L^2(0, \infty; \rho^{kN} e^{-\rho/2} d\rho)$,  the  $L^2$-space weighted by  $\rho^{kN} e^{-\rho/2} d\rho$. In addition, they span the functions $\rho^n$, $n \in \mathbf N$. The fact that  $\{P_{k, l} \}_{l \in \{0\} \cup \mathbf N} $ is a maximal orthogonal set in $L^2(0, \infty; \rho^{kN}e^{-\rho/2}d\rho)$
follows from the theorem p. 333 in \cite{sznagy;orthogonal-expansions}  which is recalled below. 

\begin{theorem}[Theorem p. 333 \cite{sznagy;orthogonal-expansions}] 
\label{thm:bela}
The orthonormal sequence of polynomials attached to a mass distribution function $\mu(x)$ on the interval $(a, b)$, is complete in $L^2(a, b; \mu)$ whenever  there exists a number $r>0$ such that the integral 
	\[
	\int_a^b e^{r|x|} d\mu(x) 
	\]
exists.
\end{theorem}

\begin{proof}[Proof of Theorem \ref{thm:completeness}]
It suffices to show that the set of all finite linear combinations of members of $S$ is dense. Let $f  (r, \theta) \in H_{Sym}(\Omega_{0, N})$. Since $f$ is periodic in $\theta$ and $f(r, \cdot) \in L^2(-\pi/N, \pi/N)$ for almost every $r$, we can expand $f(r, \cdot)$ in Fourier series for almost every  $r \in (0, \infty)$. The symmetries of $f$ imply  $f(r, \theta) = \sum_{j \in \mathbf{N}} a_j(r) \sin(jN \theta)$. Using the change of variable $\rho =r^2$, we write
	\begin{equation}
	\tilde{f} (\rho, \theta) = f(\sqrt \rho, \theta) = \sum_{j \in \mathbf{N}} b_j(\rho) \sin(jN\theta) \textrm { for a.e $\rho \in (0, \infty)$}.
	\end{equation}
The $H$-norm of $f$ is finite, therefore
	\begin{align*}
	\lefteqn{\int_{\Omega_{0, N}} |f (r, \theta)|^2  e^{-|\xi|^2/2} d\xi = \int_{r=0}^{\infty} \int_{\theta = -\pi/N}^{\pi/N} |f(r, \theta)|^2 d\theta \ e^{-r^2/2} r \ dr }\\ 
	 &=  \frac{\pi}{2N} \int_{\rho=0}^{\infty} \sum_{j \in \mathbf N} b_j^2(\rho) e^{-\rho/2} d\rho = \frac{\pi}{2N}\sum_{j \in \mathbf N} \int_{\rho=0}^{\infty}  b_j^2(\rho) e^{-\rho/2} d\rho <\infty.
	\end{align*}
Hence, the functions $b_j$'s  are in $L^2(0, \infty; e^{-\rho/2} d\rho)$ for $j \in \mathbf N$.  For every $\eps>0$, there is a $J (\eps)\in \mathbf N$ so that 
	\[
	\frac{\pi}{2N}\sum_{j > J(\eps)} \int_{\rho=0}^{\infty}  b_j^2(\rho) e^{-\rho/2} d\rho \leq \eps /2.
	\]
We can approximate $b_{j} \rho^{-jN/2} \in L^2(0, \infty; \rho^{jN} e^{-\rho/2} d\rho)$ by a linear combination of $P_{j, l}$, since $\{P_{j, l} \}_{l \in \{0\} \cup \mathbf N} $ is a maximal orthogonal set . More precisely, for every $\eps>0$ and every $j$,  there is a linear combination of $P_{j, l}$'s denoted by $Q_{j, \eps}$ so that
	\[
	\frac{\pi}{2N} \int_0^{\infty}|b_j\ \rho^{-jN/2} - Q_{j, \eps}|^2 \rho^{jN} e^{-\rho/2} \ d\rho < 2^{-j}\eps.
	\]
Define $\mathcal Q_{\eps}$ to be the linear combination of elements of $S$  given by 
	$
	\mathcal Q_{\eps} =  \sum_{0<j \leq J(\eps)} Q_{j, \eps}  \rho^{jN/2} \sin(jN \theta)$.  Then
	\begin{align*}
	\lefteqn{\| f - \mathcal{Q}_{\eps} \|^2_{H(\Omega_{0, N})} }\\
	&= \frac{\pi}{2N}  \int_{\rho=0}^{\infty}\left( \sum_{j >J(\eps) } b_j^2(\rho) e^{-\rho/2} 
	  + 
	  \sum_{0<j \leq J(\eps) } |b_j - Q_{j, \eps} \rho^{jN/2}|^2 e^{-\rho/2}\right) d\rho\\
		& \leq \eps/2 + \sum_{0<j \leq J(\eps)} 2^{-j-1} \eps \leq \eps.
	\end{align*}
This shows that every function in $H_{Sym}(\Omega_{0, N})$ can be approximated arbitrarily closely by a linear combination of functions of $S$. 
\end{proof}


\subsection{Poincar\'e inequality}
As an immediate consequence of Theorem \ref{thm:completeness}, we have the following inequality for $u \in C_c^{\infty}(\Omega_{0, N})$:
	\[
	\| u \|_V^2 = -( \Delta u -\xi \cdot Du - u, u)_H \geq - \lambda_0 \|u\|_H^2 = (N+1) \| u\|_H^2, 
	\]
where $\lambda_0$ is the lowest eigenvalue of the operator $- \lin_-$. This implies a Poincar\'e inequality
	\begin{equation}
	\label{eq:poincare}
	\| u \|_{V(\Omega')} \geq \sqrt{N+1}\ \| u \|_{H(\Omega')}, \quad u \in V_0(\Omega')
	\end{equation}
for every domain  $\Omega' \subset \Omega_{0, N}$.

\subsection{Dirichlet problem}

Let $\OmRN$ be defined as in \eqref{eq:OmRN}. 
\begin{lemma}
\label{lem:linear-problem-solvable}
For any $g \in H(\OmRN)$, the equation 
	\begin{equation}
	\label{eq:linear-problem}
	\lin (u) = \Delta u -\xi \cdot Du +u = g
	\end{equation}
possesses a weak solution $u \in V_0(\OmRN)$.
\end{lemma}
	
\begin{proof} It is a classical variational argument for the functional
	$
	J(u) = \int_{\OmRN} \Bigl\{\frac{1}{2}|D u(\xi)|^2 - \frac{1}{2}u(\xi)^2+g(\xi)u(\xi)\Bigr\}\;dm(\xi), $ over  $u\in V_0(\OmRN)$ 
and using the Poincar\'e inequality \eqref{eq:poincare}.
\end{proof}


\subsection{Maximum Principle.}

Let $\Omega'$ be any domain that is a (not necessarily proper) subset of $\OmRN$. 

\begin{definition}
We say that $u \in V(\Omega')$ satisfies $u \leq 0$ on $\partial \Omega'$ if its positive part $u^+=\max (u, 0) \in V_0(\Omega')$. 
\end{definition}
\begin{theorem} [Maximum Principle]
\label{thm:max-principle}
 Let $N \geq 5$. Suppose $u \in V(\Omega')$ satisfies 
	\begin{gather*}
	\lin u = \Delta u - \xi \cdot D u +u \geq 0 \textrm{ in  }\Omega'\\
	u \leq 0 \textrm{ on }\partial \Omega',
	\end{gather*}
then $u \leq 0$ in $\Omega'$. 
\end{theorem}  

\begin{proof}
First note that $\psi:= r\cos \theta$ is a solution of $\lin \psi =0$. By the hypotheses, we have
	\[
	\int_{\Omega'} Du\cdot  Dv - uv \ dm \leq 0, \quad \textrm{ for } v\geq 0, v \in C^1_0(\Omega').
	\]
Define the function $\varphi$ such that  
	$
	u=\psi \varphi.
	$ 
Since $1/\psi$ and $|D\psi|$ are uniformly bounded, $\varphi \in V$. Moreover, $\varphi$ satisfies
	\[
	\int_{\Omega'} \psi D\varphi \cdot  Dv -v D\psi \cdot D\varphi +D\psi \cdot D(\varphi v) - \psi \varphi v \ dm \leq 0.
	\]
Since $\lin \psi =0$ and $\varphi v \in V_0$, we have
	$
	\int_{\Omega'} D\psi \cdot D(\varphi v) - \psi \varphi v \ dm =0.
	$
Hence,
	\begin{equation} 
	\label{eq:phi}
	\int_{\Omega'} \psi D\varphi \cdot Dv - v D\varphi \cdot D\psi \ dm\leq 0.
	\end{equation}
Take $v=\max(\varphi, 0)$ in \eqref{eq:phi}. Note that we can do this since $u^+\in V_0$ and $\psi \neq 0$, so  $v \in V_0$. Therefore
	\[
	\int_{\Omega'} \psi Dv \cdot Dv \ dm \leq \int_{\Omega'} v Dv \cdot D\psi  \ dm
	\]
since $Dv=D\varphi$ on the support of $v$. Using $R \cos\frac{\pi}{2N} \leq \psi $ and $|D\psi| \leq 1$, we get
	\[
	R \cos\frac{\pi}{2N} \int_{\Omega'} |Dv|^2 dm \leq \int_{\Omega'} |Dv| v \ dm,
	\]
and 
	\[ 
	R \cos\frac{\pi}{2N} \|Dv\|_H^2 \leq \|Dv\|_H \|v\|_H
	\]
by the Cauchy Schwarz inequality. Our Poincar\'e inequality \eqref{eq:poincare} implies that
	\[
	R \cos\frac{\pi}{2N} \|Dv\|_H^2 \leq \|Dv\|_H \frac{1}{\sqrt{N}}\|Dv\|_H.
	\]
A simple computation shows that for $N\geq 5$, $R \cos\frac{\pi}{2N} -\frac{1}{\sqrt{N}} \geq \frac{1}{\sqrt 2} \cos\frac{\pi}{2N} -\frac{1}{\sqrt{N}} > 0$. Therefore
	\[
	\|Dv\|_H =0.
	\]
But $v \in V_0(\Omega')$, so $v \equiv 0$; this means that $\varphi \leq 0$ and $u\leq 0$ in $\Omega'$.
\end{proof}

\subsection{Boundary value problem}
\label{ssec:boundary-value-problem}

Let $F \in C^4([0, 2 \pi])$ with the symmetries \eqref{eq:symmetries-f}: $F(\theta) = - F(-\theta) =  F(\pi/N - \theta)$. We will use the notation $K_0 = \| F\|_{C^4([0, 2\pi])}$ throughout Section \ref{sec:linear-operator}.

\begin{definition}
\label{def:boundary-value-problem}
We say that $u \in V(\Omega)$ $(V(\OmRN))$ is a solution to the problem 
	\[
	\lin u = 0, \textrm { in  }\Omega\  (\OmRN \text{ resp.}), \quad
	u = F \textrm{ on } \pd\Omega \ (\pd \OmRN \text{ resp. })
	\]
if $u$ satisfies 
	\[
	\lin u =0, \textrm{ in } \Omega\  (\OmRN \text{ resp.}), \quad
	u - \psi \in V_0(\Omega)\ (V_0(\OmRN) \text{ resp.}),
	\]
where $\psi$ is a function in $C^4(\bar \Omega)$ ($C^4(\bar \Omega_{R,N})$ resp.) with  $\psi(R, \theta) = F(\theta)$ 
and $\psi(r, \theta) = -\psi(r, -\theta)= \psi(r, \pi/N -\theta) $ for $r \geq R,  \theta \in [0, 2 \pi)$, ($r \geq R,  \theta \in [-\pi/N, \pi/N]$ resp.). 
\end{definition}

It follows from Lemma \ref{lem:linear-problem-solvable} and the Maximum Principle (Theorem \ref{thm:max-principle}) that the solution $u$ to $\lin u =0 $ in $\OmRN$, $u=F$ on $\pd \OmRN$ exists and is well defined, in other words it is unique and independent of the choice of $\psi$. 

Denote by $U_1$ the extension of $u$ to $\Omega$ given by 
	\begin{gather*}
	U_1 (r, \theta) = u(r, \theta) \quad \textrm{ in } \OmRN \\
	U_1 (r, \theta) = - U_1 (r, -\theta) = U_1(r, \pi/N -\theta) \quad \textrm{ in } \Omega.
	\end{gather*}
The function $U_1$ is a  solution to boundary value problem 
	\begin{equation}
	\label{eq:U1-equation}
	\lin U_1 = \Delta U_1 - \xi \cdot DU_1 +U_1 =0 \textrm{ in } \Omega, \quad U_1 = F \textrm { on } \pd \Omega.
	\end{equation}
Moreover, it is smooth by standard elliptic theory.

\subsection{Behavior of solution at infinity}
\label{sec:behavior-infty}

Let $u$ be the solution to the problem $\lin u =0 $ in $\OmRN$, $u=F$ on $\pd \OmRN$. We can choose $\psi$ in Definition \ref{def:boundary-value-problem} satisfying 
	$
	\|\psi\|_{\infty} \leq \|F\|_{\infty} $ with $ \psi = 0 $ in $ \OmRN\setminus B_{R+2}.
	$
Let $a=3 \| F \|_{\infty}$ and  consider the function 
	\[
	w_1=u-a r \cos \theta.
	\]
Since $N\geq 5$, $a = 3 \| F \|_{\infty} \geq \frac{ 2 \|F\|_{\infty}}{\cos(\pi/N)}$ therefore  $w_1\leq 0$ on $\partial \OmRN$. It follows from the Maximum Principle  that $w_1\leq 0$ in $\OmRN$, hence
	$
	u (r, \theta) \leq a r \cos \theta \textrm { in } \OmRN.
	$
An similar argument with $w_2=-a r \cos \theta -u$ gives a lower bound on $u(r, \theta)$, so
	\begin{equation}
	\label{eq:linear-growth-OmRN}
	\lvert  u(r, \theta) \rvert \leq ar \cos (\theta) \leq 3 K_0 \ r \textrm { on } \OmRN.
	\end{equation}
Therefore, we also have
	\begin{equation}
	\label{eq:U1-growth}
	|U_1 (\xi)| \leq 3 K_0 |\xi|, \quad \xi \in \Omega.
	\end{equation}



\subsection{Bounds on the first and second derivatives of the solution} Let $U_1$ be the solution to \eqref{eq:U1-equation} from section \ref{ssec:boundary-value-problem}.

\begin{lemma}\label{lem:U1-bounds} There is a constant $K_1$ independent of $U_1$ and $F$ so that
	\begin{equation}
	\label{eq:U1-bounds}
	|DU_1|+|\xi||D^2U_1| \leq K_1 \|F\|_{C^4[0, 2\pi]}
	\end{equation}
\end{lemma}

\begin{proof}

\emph{Step 1: Estimate away from the boundary.}\\
The function $U_1$ is smooth in $\Omega$ by elliptic theory.  A computation shows that the function 
	\[
	v(x, t)=\sqrt{1-t}\  U_1\left(\frac{x}{\sqrt{2(1-t)}}\right)
	\] 
satisfies the heat equation
	$
	v_t=\Delta v.
	$
The theory for parabolic equations gives  us the following estimates on the derivatives of $v$,
	\begin{equation}
	\label{eq:parabest}
	|D^m v(p,0)|\leq \frac{C(m)}{s^m} \sup_{B_s(p) \times (-s^2,0)} |v(x,t)|,
	\end{equation}
where $C(m)$ is a constant independent of $v$ or $s$ and         as long as $v$ exists in $B_s(p) \times (-s^2,0)$. This is true provided 
	\[
	\frac{\min_{x\in B_s(p)} |x|}{\max_{t\in (-s^2, 0)} \sqrt{2 (1-t)}}=\frac{|p|-s}{\sqrt{2(1+s^2)}} >R.
	\]
In particular, it is true  if  $|p|>\frac{5\sqrt{2}}{3} R$ and $s=|p|/10$. From \eqref{eq:parabest}, we get
	\[
	|Dv(p,0)| \leq \frac{10\ C(1)}{|p|} \sup_{B_s(p) \times (-s^2,0)} |v(x,t)|.
	\]
The last factor of the right hand side can be estimated using equation \eqref{eq:U1-growth} to obtain
	\begin{align*}
	\sup_{B_s(p) \times (-s^2,0)} |v(x,t)| 
	&\leq\sqrt{2}\ K_0 (|p|+s).
	\end{align*}
Using the notation  $\xi=p/\sqrt{2}$, we have
	\[
	|DU_1(\xi)|=\sqrt{2}|Dv(p,0)|
	\leq 22\  C(1) K_0 	\textrm { for  } |\xi|>5R/3.
	\]
Similarly, we get an estimate on the second derivative
	\[
	|D^2 U_1(\xi)|=2|D^2 v(p,0)| \leq 
	\frac{ 220 \ C(2)}{|\xi|} K_0 \textrm { for  } |\xi|>5R/3.
	\]
Therefore
	\begin{equation}
	\label{eq:away}
	|DU_1|+|\xi||D^2U_1|\leq C \|F\|_{L^{\infty}} \textrm{ for } |\xi| >5R/3.
	\end{equation}
Since $R<  2$, the estimate \eqref{eq:away} is valid for $|\xi| \geq 4$ in particular.\\

\emph{Step 2: Estimates up to the boundary.}\\
Consider the equation 
	$
	\lin U_1 =0 \textrm { in the annulus  } A_{R, 4}=B_4 \setminus \bar{B}_R.
	$
The domain $A_{R, 4}$ is bounded so global regularity results for elliptic equations from \cite{gilbarg-trudinger;elliptic-equations} imply that
	\begin{equation}
	\label{eq:global-reg}
	\| U_1 \|_{W^{4, 2}(A_{R, 4})} \leq C ( \| U_1 \|_{L^2( A_{R, 4})}  + \| \psi \|_{W^{4, 2}( A_{R, 4})}),
	\end{equation}
where $\psi$ is a function in $W^{4,2}(A_{R, 4})$ such that $\psi -U_1 \in W^{1,2}_0(A_{R, 4})$. 

Note that we can bound the first four derivatives of $U_1$ for $|\xi| =4$ by an argument similar to the one in step 1,
	\[
	|D^j U_1| (\xi) \leq C \|F\|_{L^{\infty}} \textrm { for  } i =1, \ldots, 4 \text{ and }\xi \in \pd B_4.
	\]
It is then clear that we can find a constant $C$ independent of $F$ and $U_1$, and a function $\psi$ for which \begin{gather}
	\psi|_{\pd B_R}=F,
	\qquad 
	\psi|_{\pd B_4}=U_1|_{\pd B_4}, \notag\\
	 \|\psi \|_{C^4(A_{R, 4})} \leq C \|F\|_{C^4[0, 2 \pi]}. \label{req:3}
	\end{gather}
With this choice of $\psi$, equation \eqref{eq:global-reg} gives us
	\begin{equation}
	\label{eq:varphi}
	\| U_1 \|_{W^{4,2}(A_{R, 4}) }\leq C(\|\psi\|_{W^{4,2}(A_{R, 4}) } + \| U_1\|_{L^2(A_{R, 4}) }).
	\end{equation}
Using Sobolev's inequality, \eqref{req:3},  \eqref{eq:varphi} and  \eqref{eq:U1-growth}, we get
\begin{equation}
\label{eq:near}
\| U_1 \|_{C^2(A_{R, 4})} \leq C \|F\|_{C^4[0, 2\pi]}.
\end{equation}
Combining \eqref{eq:away} and \eqref{eq:near}, we obtain the desired result.
\end{proof}


\section{Finding sub and supersolutions}
This section is devoted to the construction of  a subsolution $u_-$ and a supersolution $u_+$ to the problem \eqref{eq:quasilinear}-\eqref{eq:boundary-condition}, or equivalently, to
	\begin{align}
 	\label{eq:nonlin}
 	 \lin u = \Delta u -\xi \cdot D u + u & =  \frac{D_i u D_ju D_{ij}u}{1+|Du|^2} \\
  	u|_{\pd\Omega} &= \eps F \nonumber
	\end{align}
with small $\eps$.  
Let us discuss the strategy for finding a supersolution first. Since the boundary data is of order $\eps$, we write $u = \eps U$, 
	\[
	\lin U =\eps^2 \frac{D_i UD_jUD_{ij}U}{1+\eps^2|DU|^2}, \qquad
	U|_{\pd\Omega} = F
	\]
and decompose $U$ into the two terms $U=U_1+\eps^2 U_2$, where
	\begin{align}
  	\lin U_1 &= 0 , & U_1|_{\pd\Omega}= F,\label{eq:U1}\\
  	\lin U_2 &= \frac{D_i UD_jUD_{ij}U}{1+\eps^2|DU|^2}, &U_2
 	 |_{\pd\Omega}=0.\label{eq:U2}
	\end{align}
We discussed the existence of such  a function $U_1$ in section \ref{sec:behavior-infty}. Choosing an appropriate function $U_2$ is done below.


\subsection{Preliminary computations for finding a supersolution}

In order to get a supersolution, we want $U_2$ to satisfy
	$
	\lin U_2 -\frac{D_i UD_jUD_{ij}U}{1+\eps^2|DU|^2} \leq 0.
	$
Roughly estimating $U$ by $U_1$ and bounding the derivatives  of $U_1$ using Lemma \ref{lem:U1-bounds}, we are looking for  a $U_2$ so that
	\[
	\lin U_2 -\frac{D_i UD_jUD_{ij}U}{1+\eps^2|DU|^2} \leq \lin U_2 + C_1^3/ r \leq 0,
	\]
where $C_1 = K_1  \|F\|_{C^4[0, 2\pi]}$. Define
	$
	v_a = a\left(r-\frac{R^2}{r}\right).
	$
A simple computation gives
	\[
	\lin v_a= -\frac{a R^2}{r^3} + \frac{a}{r}(1-2R^2), \qquad v_a =0 \textrm{ on }\pd \Omega.
	\]
This makes $v_a$ a perfect candidate for $U_2$.

\begin{remark} The term $a(1-2R^2)/r$ has to counterbalance the contribution $C_1^3/r$ from the nonlinear term. Therefore, we need $(1-2 R^2)$ to be negative. This  justifies the imposed lower bound   $R>1/\sqrt 2$. 
\end{remark}

\subsection{Existence of a supersolution and a subsolution}

Let $\eps >0$ be a small constant. 
Consider the function 
	\[
	u_+= \eps (U_1 + \eps^2 v_a),
	\]
where $a$ is a constant to be chosen later. 
	\begin{equation}
\el u_+  
	=-\eps^3 \lin v_a
- \eps^3 \frac{D_i(\Uplus) D_j(\Uplus) D_{ij}(\Uplus)}{1+\eps^2|D(\Uplus)|^2}.\label{eq:4}
	\end{equation}
We have $
	|D_i v_a|  \leq 2a$ and $
	|D_{ij} v_a|  \leq \frac{6a}{r} $, which combined with \eqref{eq:U1-bounds}, give us
%
%
	\begin{multline*}
	\frac{1}{\eps^3} \el  u_+ \leq -\frac{aR^2}{r^3} + \frac{a}{r} (1-2R^2) + (C_1 + \eps^2 2a)^2 \left(\frac{C_1}{r} + \eps^2 \frac{6a}{r} \right)\\
	= \frac{1}{r}\left( -\frac{aR^2 }{r^2} -a(2R^2 -1) +C_1^3 + 10 a\eps^2 C_1^2 +28 a^2 \eps^4 C_1+24a^3 \eps^6 \right).
	\end{multline*}
Here we used the notation  $C_1 = K_1 \| F\|_{C^4}$.
We replace $a$ by $a= k C_1^3$, with $k$ a constant to be chosen later, to get
	\begin{multline}
	\label{eq:Eu1}
	\frac{r}{\eps^3} \el u_+ \leq C_1^3 \left( \frac{-kR^2}{r^2} -k(2R^2-1)+1 \right) \\
	 +C_1^3 \left( 10 \eps^2 k C_1^2+ 28 (\eps^2 k C_1^2)^2+ 24(\eps^2 k C_1^2)^3\right).
	\end{multline}
 If $\eps$ is small enough so that 
$
\eta := \eps^2 k C_1^2  \leq 1,
$
the equation \eqref{eq:Eu1} becomes
	\[
	\frac{r}{\eps^3} \el u_+ \leq C_1^3 \left( \frac{-kR^2}{r^2} -k(2R^2-1)+1+62 \eta \right).
	\]
Hence, to have a supersolution, it suffices to find $k$ and $\eps$ that satisfy
	\[
	\eta := \eps^2 k C_1^2  \leq 1 \textrm { and } -k(2R^2-1)+1+62 \eta  \leq 0. 
	\]
If we  take $R_0$ so that  $R\geq R_0> 1/\sqrt 2$, both inequalities are true for $k = \frac{2}{2R_0^2 -1}$ and $\eps^2 C_1^2 \leq  \frac{1}{62 k} =\frac{2R_0^2 -1}{124}$.
%
Recall that $C_1 = K_1 \|F\|_{C^4}$, therefore, for $\eps$ and $k$ so that
	\[
	\eps \|F\|_{C^4} \leq \frac{1}{K_1} \sqrt{ \frac{2R_0^2 -1}{124}}, \quad  k = \frac{2}{2R_0^2 -1}, \text{ and } R\geq R_0> 1/\sqrt 2
	\]
the function $u_+ = \eps (U_1 + \eps^2 kC_1^3 (r-\frac{R^2}{r}))$ is a supersolution $\el u_+ \leq 0$ with boundary condition $u_+|_{\pd \Omega} = \eps F$. A similar argument shows that $u_- = \eps (U_1 - \eps^2 kC_1^3 (r-\frac{R^2}{r}))$ is a subsolution  with boundary condition $u_-|_{\pd \Omega} = \eps F$. If we denote by $f= \eps F$ and $u_1 = \eps U_1$, we just proved the following result.

\begin{lemma}
\label{lem:subsuper} Assume $R\geq R_0>1/\sqrt 2$ and denote by $\Omega = \R^2 \setminus \bar B_R$, the plane with a hole of radius $R$ at the origin. Let  $f $ be a function in  $C^4[0,2\pi]$ with the symmetries \eqref{eq:symmetries-f},  
 $u_1$  a solution to the linear equation $\lin u_1 =0$, $ u_1|_{\pd\Omega}= f$ and 
  $K_1$  a constant (independent of $f$) so that
	\[
	|Du_1| +|\xi| |D^2u_1| \leq K_1 \|f\|_{C^4}.
	\]
If $\eps= \|f\|_{C^4[0, 2\pi]}$ satisfies
	\[
	0< \eps= \| f \|_{C^4} \leq \frac{1}{K_1} \sqrt{ \frac{2R_0^2 -1}{124}}, 
	\]
then the function $u_+ = u_1 + K_1^3 \eps^3 k (r-\frac{R^2}{r})$, where $k = \frac{2}{2R_0^2 -1}$, satisfies
\begin{gather*}
\el u_+ = g^{ij}(D u_+) D_{ij} u_+ -\xi \cdot D u_+ +u_+ \leq 0 \textrm { in } \Omega\\
u_+ =  f \textrm{ on }  \pd \Omega, 
\end{gather*}
and the function $u_-= u_1 - K_1^3 \eps^3 k (r-\frac{R^2}{r})$ satisfies
\begin{gather*}
\el u_- = g^{ij}(D u_-) D_{ij} u_- -\xi \cdot D u_- +u_- \geq 0 \textrm { in } \Omega\\
u_- =  f  \textrm{ on }  \pd \Omega.
\end{gather*}
\end{lemma}


From the bounds on the derivatives of $U_1$ and the definition of $u_+$ and $u_-$, we have
\begin{corollary} 
\label{cor:bound-subsuper}
Let $u_+$ and $u_-$ be defined as in the lemma above. Then there exists a constant $K_2$ depending only on $ \| f \|_{C^4}$ and $R_0$ so that 
	\[
	|\xi|^{-1} |u_{\pm}| + |Du_{\pm}|  + |\xi| |D^2 u_{\pm}| \leq K_2, \quad \xi \in \Omega
	\]
\end{corollary}


\section{The Parabolic Equation $\pd_{\tau} u = \el u$}
\label{sec:parabolic-equation}

In this section, we prove that there exists a solution to the parabolic equation $\pd_{\tau} u = \el u$ for all time $\tau \in [0, \infty)$ for well chosen initial and boundary conditions $u_0$ and $f$ respectively.

To avoid confusion, let us fix notations: we denote by $D_i$ the ordinary derivative with respect to $i$-th coordinate $x_i$ or $\xi_i$ of a spatial variable and by $D v$ the gradient of the function $v$ with respect to spatial variables. The notation $\pd_{\tau}$ or $\pd_t$ is reserved for the derivative with respect to time.

%
 \subsection{Initial condition and short time existence}
 
 For the initial condition to the parabolic equation, we  choose a $C^3$ function $u_0$ in $\Omega$ that stays between the  sub and supersolutions constructed in Lemma \ref{lem:subsuper} and  has bounded derivatives
 	\begin{gather*}
	u_- \leq u_0 \leq u_+,\\
	\|D_{\xi} u_0 \|_{C^{2}( \bar \Omega)} \leq M_0< \infty. 
	\end{gather*}
Moreover, we assume that $u_0$ satisfies the following symmetries and compatibility conditions
 	\begin{gather}
	u_0(r, \theta) = - u_0 (r, -\theta) = u_0 (r, \pi/N - \theta), \label{eq:symmetries}\\
	\label{eq:comp-cond}
	u_0 (R, \theta) = f(\theta), \quad \el(u_0) (R, \theta) =0.
	\end{gather} 
Let us define the domains $\mathcal O_{t} = \{ x \in \R^2\mid |x| > \sqrt{2(1-t)} \ R\}$ and  $\mathcal Q_T = \{ (x, t) \mid 0<t<T , |x| >\sqrt{2(1-t)} \ R\}$. We consider the following corresponding problem for the mean curvature flow
	\begin{subequations}
	\label{eq:MCF-problem}
	\begin{gather}
	\label{eq:MCF}
   	\pd_t v(x, t)  = g^{ij}(D v)(x, t)D_{ij} v(x, t) ,    \\
	v(x, 0) = v_0 (x),  x \in \mathcal O_0 \label{eq:MCF-boundary1}\\
	v(x, t) = \sqrt 2 f (\frac{x}{\sqrt{2 (1-t) }}),    t \geq 0, x \in \pd \mathcal O_t.\label{eq:MCF-boundary}
	\end{gather}
	\end{subequations}
for some initial condition $v_0(x)$. Standard parabolic theory assures the existence of a smooth short time solution $v(x, t)$ to the boundary value problem \eqref{eq:MCF-problem} with initial condition $v_0(x) = \sqrt 2 u_0(\frac{x}{\sqrt 2})$. Moreover, if we denote the maximal time interval in which $v$ exists by $[0, t_0)$, we have
	\begin{equation}
	\label{eq:Dv-finite}
	M_t=  \sup_{\mathcal O_t} |Dv| + \sup_{\mathcal O_t} |D^2 v| < \infty, \quad t \in [0, t_0).
	\end{equation}
 For more details, we refer to \cite{eidelman;parabolic-systems}. 
 
 In order to show that the solution $v$ exists for all time $t$, we argue by contradiction and assume that $t_0<\infty$ then show that $M_t$ is uniformly bounded for all $t \in [0, t_0)$. Standard parabolic theory implies that the solution $v$ can then be continuated past $t_0$, which contradicts the maximality of $[0, t_0)$.  The rest of this section is devoted to the proof of a uniform bound on $M_t$.
 
 The function $u$ defined by
	\begin{equation}
	\label{eq:definition-v}
        v(x,t) = \sqrt{2(1-t)} \; u\bigl(\frac{x}{\sqrt{2(1-t)}}, -\frac{1}{2}\ln(1-t)\bigr)
	\end{equation}
satisfies
	\begin{subequations}
	\label{eq:zoom-in-problem}
	\begin{gather}
	\label{eq:zoom-in}
   	\pd_\tau u = g^{ij}(D u)D_{ij} u - \xi\cdot D u+u \textrm{ in } \Omega \times [0, \tau_0)\\
	u(\xi, 0)=u_0(\xi), \xi \in \Omega, \qquad
	u(\xi, \tau) = f  \text{ on } \pd \Omega\times [0, \tau_0), \label{eq:boundary-cond}
	\end{gather}
	\end{subequations}
where we used the change of variables $\xi = \frac{x}{\sqrt{2(1-t)}}, \tau = -\frac{1}{2} \ln(1-t)$ and $\tau_0 =  -\frac{1}{2} \ln(1-t_0)$. 

\begin{lemma} 
\label{lem:bound-compact-time}
Let $u$ be a smooth solution of  \eqref{eq:zoom-in}  in $\Omega \times [0, \tau_0)$ with boundary condition \eqref{eq:boundary-cond}, then 
	\begin{equation}
	\label{eq:bound-compact-time}
	\max_{\Omega \times [0, \tau_1]} |D_{\xi}u| +
	\max_{\Omega \times [0, \tau_1]} |D_{\xi\xi}^2u| < M(\tau_1) \textrm{ if } \tau_1<\tau_0
	\end{equation}
\end{lemma}
\begin{proof}
To each solution $u$ to \eqref{eq:zoom-in-problem} corresponds a solution $v$ to  the problem \eqref{eq:MCF-problem} via the formula \eqref{eq:definition-v}. This lemma is therefore a consequence of \eqref{eq:Dv-finite}.
\end{proof}


\subsection{Bounds on the function $u$}

Here, we show that  a solution $u$ to the problem \eqref{eq:zoom-in-problem}  stays between the subsolution $u_-$ and the supersolution $u_+$:
\begin{lemma} 
\label{lem:bound-u}
Let $u(\xi, \tau)$ be a $C^2$ solution to the problem \eqref{eq:zoom-in-problem}, then
	\begin{equation}
	\label{eq:squeezeu}
	u_-(\xi) \leq u(\xi, \tau) \leq u_+(\xi), \quad |\xi| \geq R, 0 \leq \tau < \tau_0.
	\end{equation}
	\end{lemma}

\begin{proof} The proof is a slight modification of the proof of a Maximum Principle for parabolic equations. Let $K_2$ be as in Corollary \ref{cor:bound-subsuper}. For small $\eps>0$,  consider the function $W = \psi (u_--u) - \eps$, where $\psi(\xi, \tau) := e^{-\eps |\xi| - (1+\delta)\tau}$ and $\delta$ is to be chosen later. Note that $W$ becomes negative as $|\xi|$ tends to infinity. We have
	$
	\pd_{\tau} \psi = -(1+\delta) \psi, 
	D_i \psi = -\eps \frac{\xi_i}{|\xi|} \psi,$ and $
	D_{ij} \psi =\Bigl[-\frac{\eps}{|\xi|} (\delta_{ij} - \frac{\xi_i\xi_j}{|\xi|^2}) + \eps^2 \frac{\xi_i \xi_j}{|\xi|^2}\Bigr] \psi.
	$ 
A computation shows that $W$ satisfies 
	\begin{multline} 
	\label{eq:pdtauW}
	\pd_{\tau} W \leq - \delta  (W+\eps) - \psi \xi \cdot DW - \eps |\xi| (W+\eps) \\
	+ \psi \bigl(g^{ij}(Du_-) - g^{ij} (Du) \bigr) D_{ij} u_- \\
	+ g^{ij} (Du) \Bigl( D_{ij} W + 2 \eps \frac{\xi}{|\xi|} \cdot DW + (W+\eps) \frac{ D_{ij} \psi}{\psi} \Bigr).
	  \end{multline}
Fix $\tau_1 \in [0, \tau_0)$. The estimates in Corollary \ref{cor:bound-subsuper} and Lemma \ref{lem:bound-compact-time} guarantee the existence of a large constant $\eta$ such that $e^{-\eps |\xi|} (|u_-|+|u|) - \eps \leq -\eps /2$ for $|\xi| \geq \eta$, $\tau \in [0, \tau_1)$. 
Denote by $A_{R, \eta}$ the annulus $\{ R < |\xi| <\eta \}$. On the boundary of the cylinder $A_{R, \eta} \times [0, \tau_1]$, we have
	\begin{align*}
	&W(\xi, \tau) 
	\leq -\eps/2 \quad & |\xi| = \eta, \tau \in [0, \tau_1]\\
	&W(\xi, \tau) \leq -\eps \quad  &|\xi| = R, \tau \in [0, \tau_1]\\
	&W(\xi, 0) \leq - \eps  \quad  &A_{R, \eta} \times \{0\}.
	\end{align*}
Suppose that $W>0$ for some point in the cylinder $A_{R, \eta} \times [0,\tau_1]$. There is a first time $\tau_2 \in (0, \tau_1)$ for which  there is a point $\hat \xi$ with  $W(\hat \xi, \tau_2) = 0$. At this point $(\hat \xi, \tau_2)$, 
	 $ 
	D W(\hat \xi, \tau_2) = 0, 
	- D^2 W(\hat \xi, \tau_2) $ 
	 is positive definite, and 
	$\pd_{\tau} W(\hat \xi, \tau_2) \geq 0$. 
 At $(\hat \xi, \tau_2)$, the equation \eqref{eq:pdtauW} yields
	\begin{multline*}
	0 
	\leq -\eps \delta - \eps^2 |\hat \xi| + \psi \bigl(g^{ij}(Du_-) - g^{ij} (Du) \bigr) D_{ij} u_- + \eps g^{ij}(Du) \frac{D_{ij} \psi}{\psi}.
	\end{multline*}
The last term can be estimated by
	$ 
	\eps g^{ij} \frac {D_{ij} \psi}{\psi}  
		 \leq \frac{\eps ^2}{|\xi|} + \frac{\eps ^2 }{|\xi|} + \eps^3 \leq 4\eps^2.
	$ 
To bound the penultimate term, note that 
	\begin{align*}
	|\pd_{p_k} g^{ij} (p) | &= \left \lvert -\frac{ \delta_{ik} p_j + \delta_{jk} p_i}{1+|p|^2} + \frac{4 p_i p_j p_k }{(1+|p|^2)^2} \right \rvert \leq 6\label{eq:pdg}\\
	(Du_- - Du) (\xi_2, \tau_2)&= \frac{\eps D\psi }{\psi^2 }(\xi_2, \tau_2)
	\end{align*}	
therefore,
	$
	\left\lvert \psi \bigl(g^{ij}(Du_-) - g^{ij} (Du) \bigr) D_{ij} u_- \right\rvert \leq 12 \sqrt 2 \eps^2 K_2$ at $ (\xi_2, \tau_2)$ by the Mean Value Theorem. We obtain
	\begin{align*}
	0  \leq -\eps \delta - \eps^2 |\xi_2| + 12\sqrt 2 \eps^2 K_2 + 4 \eps ^2  \leq \eps (-\delta + 12\sqrt 2\eps K_2+4\eps).
	\end{align*} 
The choice of $\delta = 18\eps K_2+4 \eps$ makes the last expression negative. This is not possible. We have therefore shown that $W (\xi, \tau) \leq 0$ for $ R \leq  |\xi| \leq \eta, \tau \in [0, \tau_1]$. The argument stays valid for any $\eta' >\eta$, thus 
	\[
	u_-(\xi, \tau) - u(\xi, \tau) \leq \eps e^{\eps |\xi| + (1+ 18 K_2 \eps + 4 \eps) \tau}, \quad (\xi, \tau) \in \Omega \times [0, \tau_1].
	\]
Letting $\eps$ tend to $0$, we obtain
	$
	u_-(\xi, \tau) \leq u(\xi, \tau),  $ for $(\xi, \tau) \in \Omega \times [0, \tau_1].
	$
Since $\tau_1$ is arbitrary, we have 
	\[
	u_-(\xi, \tau) \leq u(\xi, \tau),  \text{ for }(\xi, \tau) \in \Omega \times [0, \tau_0).
	\]
The inequality $u(\xi, \tau) \leq u_+(\xi, \tau)$ is proved in a similar way.
\end{proof}


\subsection{Bound on the first derivative $|Du|$} 

We prove the bound in two steps. First, we use Lemma \ref{lem:bound-u} and interior estimates from the mean curvature flow to bound $|Dv|$. The interior estimate on $|Dv|$ gives us estimates for $|Du|$ for points staying farther away as time progresses. The second step is to consider an annulus $A_{R, \eta} = \{ R < |\xi| <\eta\}$, with $\eta$ large enough to have bounds for $|Du|$ on $|\xi| = \eta$ from the first step. A maximum principle for the parabolic equation satisfied by $|Du|^2$ yields  estimates in the interior of the annulus, and therefore on the whole domain.

\begin{lemma}
\label{lem:boundDu}
Let $u(\xi, \tau)$ be a smooth solution to \eqref{eq:zoom-in}  in the cylinder $\Omega \times [0, \tau_0)$ with boundary conditions \eqref{eq:boundary-cond}. There is a constant $M_1$ depending only on the boundary conditions $u_0$ and $f$ such that
	\begin{equation}
	\label{eq:boundDu}
	|Du(\xi, \tau)| \leq M_1
	\end{equation}
for all times $\tau$ for which $u$ exists, and $|\xi| \geq R$  
\end{lemma}
\begin{proof}

\emph{ Step 1. Away from the boundary. }  \\
Let $v(x, t)$ be the corresponding solution to the MCF defined by equation \eqref{eq:definition-v}. We use the interior estimate  below established by Ecker and Huisken: 

\begin{theorem}[Theorem 2.3, p551 \cite{ecker-huisken;interior-estimates}]
\label{thm:ecker-huisken} 
The gradient of the height function $v$ satisfies the estimate
	\begin{multline} 
		\label{eq:ecker-huisken}
	\sqrt{1+|Dv(x_0, t)|^2} \leq C_1(n)\sup_{B_{\rho}(x_0)} \sqrt{1+|Dv(x, 0)|^2}\\
  	\exp[C_2(n) \rho^{-2} \sup_{[0,T]}(\sup_{B_{\rho}(x_0)\times[0, T]} v -v(x_0, t))^2] 		
	\end{multline}
where $0 \leq t \leq T, B_{\rho}(x_0)$ is a ball in $\Omega$ and $n$ is the dimension of the graph of  $v$.
\end{theorem}

The first factor is estimated by
	\begin{equation}
	\label{eq:Dv0}
	 |Dv_0| = |Du_0| < \| Du_0 \|_{C^2(\bar\Omega)} =M_0.
	\end{equation}
To bound the second factor,  we recall that 
	$
	u_- (\xi) \leq u(\xi, \tau) \leq u_+ (\xi),
	$
therefore Corollary \ref{cor:bound-subsuper} gives us
	\begin{equation}
	\label{eq:boundv}
	|v(x,t)| \leq K_2 |x|.
	\end{equation}
The function $v(x, t)$ is defined on the domain $ \{ (x, t) \mid 0\leq t<t_0, |x| \geq \sqrt{2(1-t)} \ R\}$, in particular, it is defined for  $|x| > R \sqrt{2}$ at any time $0 \leq t < t_0$. Let $|x_0| \geq 2R \sqrt 2$ and $\rho= |x_0|/4$. The estimates \eqref{eq:Dv0} and \eqref{eq:boundv} yield, 
	\begin{align*}
	\sqrt{1+|Dv(x_0, t)|^2}& \leq C \text{ for $|x_0| \geq 2\sqrt 2 R$ and $t \in [0, t_0)$}.
	\end{align*}
Note that 
	$
	Du\bigl(\frac{x}{\sqrt{2(1-t)}}, -\frac{1}{2} \ln(1-t)\bigr) = Dv(x, t).
	$
With the change of variables $\tau = -\frac{1}{2} \ln(1-t)$ and $\xi =\frac{x}{\sqrt{2(1-t)}} = \frac{x e^{\tau}}{\sqrt{2}}$, we get
\begin{equation}
\label{eq:Duaway}
|Du(\xi, \tau)| \leq K_3, \quad \textrm{ for } |\xi| \geq 2 R e^{\tau}, \tau \in [0, \tau_0).
\end{equation}
Withouth loss of generality, we can assume that $K_3 \geq M_0 = \|  Du_0 \|_{C^2(\bar \Omega)}$.\\

\emph{Step 2: Up to boundary, using a Maximum Principle.} \\
A computation shows that the function $w=|Du|^2$ satisfies the parabolic equation
	\[
	\pd_{\tau}  w \leq  g^{ij}(Du) D_{ij} w + g^{ij}_{p_l}(Du) D_{ij} u\ D_l w- \xi \cdot D w,
	\]
where $ g^{ij}(p) = \delta_{ij}- \frac{p_ip_j}{1+|p|^2}$ and  $g^{ij}_{p_l} = \frac{\pd}{\pd p_l} g^{ij}$.
Denote by $	Z_l = g^{ij}_{p_l}(Du)D_{ij} u$, then 
	\begin{equation}
	\label{eq:parabolicDu}
	\pd_{\tau} w \leq g^{ij}(Du) D_{ij} w + Z_l D_l w - \xi \cdot D w.
	\end{equation}
Choose $0< \tau_1 < \tau_0$. 
It follows from Lemma \ref{lem:bound-compact-time} that
the coefficients in \eqref{eq:parabolicDu} are bounded  on $\Omega \times [0, \tau_1]$. Moreover, from \eqref{eq:Duaway}, we know that $w$ is bounded in the cylinder $ \Omega \times [0, \tau_1]$. A Maximum Principle for parabolic equations on unbounded domains (see Theorem 8.1.4 in \cite{krylov;lectures-elliptic-parabolic} for example) implies that the function $w$ in $\Omega \times [0, \tau_1)$ is smaller than its supremum on the parabolic boundary $[ \pd \Omega \times (0, \tau_1)] \cup [ \bar \Omega \times \{0\}]$. Since $\tau_1$ is arbitrary, the result is valid on the time interval $[0, \tau_0)$ as well,
	\[
	\sup_{\Omega \times (0, \tau_0)} w \leq \sup_{[ \pd \Omega \times (0, \tau_0)] \cup [ \bar \Omega \times \{0\}]} w \leq \max(K_2^2, M_0^2). \qedhere
	\]
	
 \end{proof}

\subsection{Bound on $|D^2u|$ and $|D^3 u|$ away from boundary}

\begin{lemma} 
\label{lem:D2uaway}
Let $u$ be a solution of the problem \eqref{eq:zoom-in-problem}  in $\Omega \times [0, \tau_0)$, then we have the following bounds for  the second and third derivatives of $u$:
\begin{equation}
\label{eq:D2uaway}
|D_{\xi\xi}^2 u(\xi, \tau)| \leq e^{-\delta_0} C+  \max_{\Omega \times [0, \delta_0]} |D_{\xi\xi}^2u|, |\xi| \geq e (R+1),  \tau \in [0, \tau_0)
\end{equation}
\begin{equation}
\label{eq:D3uaway}
|D_{\xi\xi}^3 u(\xi, \tau)| \leq e^{-2\delta_0} C+ \max_{\Omega \times [0, \delta_0]} |D_{\xi\xi}^3 u|, |\xi| \geq e (R+1),  \tau \in [0, \tau_0)
\end{equation}
where $\delta_0 = \min(\tau_0/2, 1)$ and $C$  denotes different constants depending only on $\delta_0$ , $\| Du_0 \|_{C^2(\bar \Omega)}$ and $\|f\|_{C^4}$.
\end{lemma}
\begin{remark}
\label{rem:delta0}
The inequality \eqref{eq:D2uaway} does depend on $\tau_0$ through the constant $\delta_0$; however the dependence is very loose. In the proof, we will see that $\delta_0$ can be chosen to be $\min(1, s/2)$ for  any time $0<s< \tau_0$. Therefore, the estimate is uniform in time once we know that the solution exists past some small time $s$.  
\end{remark}
\begin{remark}
 $\max_{\Omega \times [0, \delta_0]} |D_{\xi\xi}^2u|$ and $\max_{\Omega \times [0, \delta_0]} |D_{\xi\xi}^3u|$ depend on $u$ on the right hand side, but we are primarily interested in finding a bound for $\tau$ approaching $\tau_0$ so there is no need for a better estimate near initial time.
\end{remark}

\begin{proof} Let us tackle the second derivative first. 
The estimate for $\tau\leq \delta_0$ is immediate. For $\tau>\delta_0$,  we use the change of variables $t = 1- e^{-2 \tau}$ and look at the function 
	\begin{equation}
	\tag{\ref{eq:definition-v}}
        	v(x,t) = \sqrt{2(1-t)} \; u\bigl(\frac{x}{\sqrt{2(1-t)}}, -\frac{1}{2}\ln(1-t)\bigr),
	\end{equation}
which satisfies the mean curvature flow \eqref{eq:MCF}. We are interested in time $t$ such that  $t>1-e^{-2\delta_0}$. We can use the following interior estimate on the curvature from Ecker-Huisken  \cite{ecker-huisken;interior-estimates} as long as $v$ exists in the domains mentioned. Define $t_0$ to be $1- e^{-2 \tau_0}$.

\begin{lemma}[Corollary 3.2 \cite{ecker-huisken;interior-estimates}]
Let  $0<t_1<t_0$ and $\eta>0$ and $0\leq\theta<1$. For $t\in [0, t_1]$, we have the estimate
\[
\sup_{B_{\theta \eta} (x_0)} |A|^2(t) \leq c(n) (1 -\theta^2) ^{-2} \left( \frac{1}{\eta^2}+\frac{1}{t}\right) \sup_{B_{\eta}(x_0)\times[0, t]} (1+|Dv|^2)^2,
\]
where $B_{\eta}(x_0)$ denotes the ball of radius $\eta$ centered at $x_0$ in the plane. 
\end{lemma}

First, we show that $v$ exists in $B_{\eta}(x_0)\times[0, t_1]$ for $\eta=1$, $\theta=1/2$ and $|x_0|\geq\sqrt{2}(R+1)$: if $(x, t)  \in B_{\eta}(x_0) \times [0,t_1]$, then 
$
|x| \geq |x_0| - \eta,
$
and the corresponding  $\xi=\frac{x}{\sqrt{2(1-t)}}$ satisfies
$ 
|\xi| \geq \frac{|x_0| - \eta}{\sqrt{2(1-t)}} \geq \frac{\sqrt{2}(R+1) - 1}{\sqrt{2(1-t)}} \geq R.
$ 
Hence, for $|x_0| >\sqrt{2}(R+1)$ and all $t\in [1-e^{-2\delta_0}, t_0)$, 
\[
|A|^2(x_0, t) \leq \sup_{B_{1/2}(x_0)} |A|^2(t) \leq C \frac{16}{9}(1+\frac{1}{1-e^{-2\delta_0}}) C \leq C(\delta_0).
\]
Bounds on the second fundamental form $A$ and on the first derivative $Dv$ yield a bound on the second derivative of $v$:
\[
|D^2_{xx} v (x_0, t)| \leq C(\delta_0), \quad |x_0| \geq \sqrt{2}(R+1), t\in [1-e^{-2\delta_0}, t_0),
\]
where $C(\delta_0)$ denotes a different constant, also dependent on $\delta_0$. Therefore
	\begin{equation}
	\label{eq:D2u-first-step}
	\begin{split}
	|D^2_{\xi\xi}u(\xi, \tau)|& = \sqrt{2}  e^{-\tau}|D^2_{xx}v( \sqrt{2} e^{-\tau} \xi, 1-e^{-2\tau})| \\
	& \leq e^{-	\tau}C(\delta_0), \quad |\xi|\geq e^{\tau}(R+1), \tau \in [\delta_0, \tau_0).
	\end{split}
	\end{equation}

In order to prove an estimate independent of time and outside a fixed annulus, note that for all $\tau_1 \geq0$, the function $\tilde u ( \xi, s) = u( \xi, s +\tau_1)$ is also a solution of $\pd_{s} \tilde u = \el \tilde u$. Moreover,  $|D_{\xi} \tilde u (\xi, s)| \leq M_1$  for $(\xi, s) \in \Omega \times [0, \tau_0 - \tau_1)$ so we have an estimate similar to \eqref{eq:D2u-first-step} in this case also
	\[
	|D^2_{\xi\xi} \tilde u(\xi, s) | \leq e^{-s}C(\delta_0), \quad |\xi|\geq e^{s}(R+1), s \in [\delta_0, \tau_0-\tau_1).
	\]
Let $s = \delta_0$ and $\tau = \tau_1 +s$, the inequality above yields
	\[
	|D^2_{\xi\xi} u(\xi, \tau) | \leq e^{- \delta_0}C(\delta_0), \quad |\xi|\geq e^{\delta_0}(R+1)
	\]
for all $\tau_1 \geq 0$, i.e. for all $\tau \geq \delta_0$.

The bound on the third derivative is proved in a similar fashion using 
	\[
	\sup_{B_{\theta \eta} (x_0)} |\nabla A|^2(t) \leq c  \left( \frac{1}{\eta^2}+\frac{1}{t}\right)^2,
	\]
with the constant $c$ depending on $\theta$ and $ \sup_{B_{\eta}(x_0)\times[0, t]} (1+|Dv|^2)$ from Theorem 3.4 in \cite{ecker-huisken;interior-estimates}.  
\end{proof}


\subsection{H\"older continuity of $D_{\xi} u$ near boundary} 

Let the time $\tau_1 <\tau_0$ and denote by 
\begin{gather*}
A_{R, R_1} \textrm{ the annulus } \{ \xi \R^2 \mid R< |\xi|<R_1\}, \\
Q_{R, R_1, \tau_1} \textrm{ the domain } A_{R, R_1} \times (0, \tau_1),\\
\Gamma_{R, R_1, \tau_1} \textrm{ the set } (A_{R, R_1} \times \{0\}) \cup( (\pd B_{R_1} \cup \pd B_{R})\times [0,\tau_1]),\\
M_1=\max_{\Omega\times [0, \tau_0)} |D_{\xi} u|
\end{gather*}
For a domain $Q \in  \R^2 \times [0, \infty)$, denote by $C^{2, 1}(\bar Q) $ the set of all continuous functions in $\bar Q$ having continuous derivatives $D_{\xi}$, $D^2_{\xi\xi} u, \pd_{\tau} u $ in $\bar Q$ and by $C^{\alpha, \alpha/2}(Q)$  the set of functions on $Q$ that are $\alpha$-H\"older continuous in $\xi$ and $\alpha/2$- H\"older continuous in $\tau$.

We use the following theorem by Lady{\v{z}}enskaja et al.

\begin{theorem} [Theorem 2.3, p533 \cite{ladyzhenskaja;linear-parabolic-type}]
\label{lem:Duholder}
Let $u$ be a solution of equation
	\begin{equation}
	\label{eq:parabolic-lady}
	\pd_{\tau} u -a_{ij}(\xi, \tau, u, D_{\xi}u) D_{\xi_i \xi_j} u + a(\xi, \tau, u, D_{\xi}u) =0 
	\end{equation} 
belonging to $C^{2, 1}(\bar Q \rrt)$. Suppose that equation \eqref{eq:parabolic-lady} is parabolic with respect to $u(\xi, \tau)$, i.e the inequality
\[
\eta |\zeta|^2 \leq a_{ij}(\xi, \tau, u(\xi, \tau), D_{\xi} u(\xi, \tau))\zeta_i\zeta_j \leq \mu |\zeta|^2, \quad \eta, \mu>0
\]
is fulfilled and the functions $a_{ij}(\xi, \tau, u, D_{\xi} u)$ and $a(\xi, \tau, u, D_{\xi} u)$ are continuous and continuously differentiable with respect to all of their arguments in the domain
\[
\left\{ (\xi, \tau)\in \bar{Q}\rrt, |u| \leq \max_{Q\rrt}|u(\xi, \tau)|, \ |p| \leq M_1=\max_{Q\rrt} |u_{\xi}(\xi, \tau)| \right\}.
\]
Suppose also that there is a constant $M_2$ so that
\begin{multline*}
\max_{i, j, k=1, 2} \max_{Q\rrt} \Bigl(|a_{ij}(\xi, \tau, u(\xi, \tau), D_\xi u(\xi, \tau))|; |a| ;\left|\frac{\pd a_{ij}}{\pd D_{\xi_k} u}\right|;\left|\frac{\pd a}{\pd D_{\xi_i} u}\right|;\\
\left| \frac{\pd a}{\pd \tau}\right|; \left|\frac{\pd a_{ij}}{\pd \tau}\right|;\left|\frac{\pd a_{ij}}{\pd \xi_k}\right|;\left|\frac{\pd a}{\pd \xi_k}\right|;\left|\frac{\pd a_{ij}}{\pd u}\right|;\left|\frac{\pd a}{\pd u}\right|\Bigr)\leq M_2,
\end{multline*}
Then for some $\alpha>0$, the quantity $\|D_{\xi} u\|_{C^{\alpha, \alpha/2}(Q_{R, R_1, \tau_1})}$  is estimated by a constant depending only on $M_1$, the quantities $\eta, \mu$ and $M_2$, the norm $\| u \|_{C^{2, 1}(\Gamma\rrt)}$ and the norm in $C^2$ of the functions defining the boundary of $A_{R, R_1}$. These same quantities also determine the exponent $\alpha$. 
\end{theorem}

The coefficients of the parabolic equation 
	\begin{equation}
	\pd_\tau u = g^{ij}(D u)D_{ij} u - \xi\cdot D u+u \tag{\ref{eq:zoom-in}}
	\end{equation}
in $Q\rrt$ satisfy the hypotheses of Theorem \ref{lem:Duholder}. We therefore have a bound on $\| D_{\xi} u\|_{C^{\alpha, \alpha/2}( \bar A\rrt \times [0, \tau_1])}$ that  depends on $\| f \|_{C^4}$, $R$, $M_1$ and $\| u \|_{C^{2, 1}(\Gamma \rrt)}$. From Lemma \ref{lem:D2uaway}, we can estimate this last quantity uniformly in time, therefore,
	\begin{equation}
	\label{eq:Duholder}
	\| D_{\xi} u\|_{C^{\alpha, \alpha/2}( \bar A_{R, R_1} \times [0, \tau_0))} \leq K_4<\infty,
	\end{equation}
where $K_4$ only depends on time through $\delta_0 = \min(1, \tau_0/2)$.

\subsection{H\"older continuity of $D_{\xi\xi} u$ near boundary}

The coefficients $g^{ij} (Du)$ of \eqref{eq:zoom-in} are  in $C^{\alpha, \alpha/2}(\bar A_{R, R_1} \times [0, \tau_0))$ hence we can apply standard theory  for linear parabolic equations with H\"older coefficients (Theorem 5.2 p320 \cite{ladyzhenskaja;linear-parabolic-type} for example) to obtain
	\begin{equation}
	\label{eq:D2uholdernear}
	\| u \|_{C^{\alpha+2, \alpha/2+1}(\bar A_{R, R_1} \times [0, \tau_0))} \leq K_5 < \infty.
	\end{equation}
The Remark  \ref{rem:delta0} explains the loose dependence of our bounds on $\delta_0$ and $\tau_0$. Let us emphasize it again: if the solution $u$ is known to exist past a small time $s>0$, we can choose  $\delta_0 = s/2$ and the constants in all the estimates  following \eqref{eq:D2uaway} do not depend on $\tau_0 > s$. 

\subsection{Continuation of the solution}

Combining \eqref{eq:D2uaway} and \eqref{eq:D2uholdernear}, we have
	\[
	\| D_{\xi\xi} u \|_{C^{\alpha, \alpha/2}(\bar \Omega \times [0, \tau_0))}+\| \pd_{\tau} u \|_{C^{\alpha, \alpha/2}(\bar \Omega \times [0, \tau_0))} \leq K_6 < \infty.
	\]
Moreover, Lemma \ref{lem:boundDu} gives us  $\| D_{\xi} u \| \leq M_1 < \infty $ on $\bar \Omega \times [0, \tau_0)$. These estimates can be transformed into uniform bounds on the first two derivatives of the corresponding solution $v$ to the mean curvature flow \eqref{eq:MCF} for time $t \in [0, t_0)$. Therefore the solution $v$ has a continuation existing for time $t \in [0, t_3)$ with  $t_3>t_0$. This contradicts the maximality of $t_0$, and proves that the solution $u$ to \eqref{eq:zoom-in-problem} exists for all time $\tau \in [0, \infty)$. Since none of the bounds on $u$, $Du$ or $D^2u$ depended on $\tau_0$, we have 
	\begin{gather}
	u_- (\xi) \leq u(\xi, \tau) \leq u_+(\xi), \quad (\xi, \tau) \in \bar \Omega \times [0, \infty) \label{eq:u-bound-all},\\
	|D_{\xi} u (\xi, \tau)| \leq M_1, \quad (\xi, \tau) \in \bar \Omega \times [0, \infty) \label{eq:Du-bound-all},\\
	\| D_{\xi\xi} u \|_{C^{\alpha, \alpha/2}(\bar \Omega \times [0, \infty))}+\| \pd_{\tau} u \|_{C^{\alpha, \alpha/2}(\bar \Omega \times [0, \infty))} \leq K_6 \label{eq:D2u-bound-all}.
	\end{gather}

\section{A Solution to the Elliptic Equation $\el u$=0}
\label{sec:elliptic}


Here, we show that there is a subsequence of $u(\cdot, \tau_n)$ that tends to a solution to the elliptic equation $\el u =0$ as $\tau_n \to \infty$.
  
Let us denote by $M_{\tau}$ the graph of $u(\cdot, \tau)$ over $\Omega$, by $X=(\xi, u(\xi, \tau))$  the position vector of a point on $M_{\tau}$. From the equation $\pd_{\tau}u = \el (u)$, the variation of $X$ with respect to time is 
	\[
	\pd_{\tau} X = (H+X \cdot \nu) \sqrt {1+|Du|^2}\ \mathbf{e}_3 =: Z,
	\]  
where $\mathbf{e}_3$ is the unit vector $(0, 0, 1)$ in $\R^3$. It will be useful for what follows to decompose $\mathbf e _3 $ and $Z$  into their tangent and normal components with respect to $M_{\tau}$:
	\begin{align*}
	\mathbf e_3 &= \mathbf v+ \frac{\nu}{\sqrt{1+|Du|^2}},\\
	Z &= Z^T + Z^{\perp} =  (H+ X \cdot \nu) \sqrt{1+|Du|^2}\ \mathbf{ v}  + (H+X\cdot \nu) \nu
	\end{align*}
where $\mathbf v = \frac{1}{1+|Du|^2}(D_1 u, D_2u, |Du|^2)$.  

Consider the functional 
	$
	J(u) = \int_{M_{\tau}} e^{-|X|^2/2} dH,
	$
where $dH$ is the Hausdorff measure on $M_{\tau}$. We have
	\begin{align*}
	\frac{d J(u)}{d\tau} 
		&=  \int_{M_{\tau}} \Big[(div_{M} Z^{\perp} + div_{M} Z^T)-  X \cdot \pd_{\tau} X\Big] e^{-|X|^2/2} \ dH  \\ 
	&= \int_{M_{\tau}} e^{-|X|^2/2} (-H (H+X\cdot \nu)) \ dH +\int_{M_{\tau} }e^{-|X|^2/2} div_M Z^T \ dH\\
	& \ \ \ \ \ -\int_{M_{\tau}} X \cdot \pd_{\tau} X e^{-|X|^2/2} dH\\
	&= -\int_{M_{\tau}} e^{-|X|^2/2}  (H+X\cdot \nu)^2\ dH +\int_{\pd M_{\tau} }e^{-|X|^2/2}  Z^T \cdot \eta \ ds,
\intertext{where $\eta$ is the outward conormal unit vector to $\pd M_{\tau}$. On the circle $r=R$,  $0=\pd_{\tau} u=  (H+X \cdot \nu) \sqrt {1+|Du|^2}$ therefore }
	\frac{d J(u)}{d\tau}
	&= -\int_{M_{\tau}} e^{-|X|^2/2}  (H+X\cdot \nu)^2\ dH \leq 0.	\end{align*}
Writing $\frac{d J(u)}{d\tau}$ in terms of $\pd_{\tau} u$ and as an integral over $\Omega$, we obtain
	\[
	\frac{d J(u)}{d\tau}=- \int_{\Omega} \frac{(\pd_{\tau} u)^2}{ \sqrt{1+|Du|^2}} e^{-\frac{|\xi|^2 +u^2}{2}}\ d\xi \leq 0.
	\]
By definition, $J(u(\cdot , \tau)) \geq 0$ for all time $\tau \in [0, \infty)$, so
	\[
	  \int_{0}^{\infty}\int_{\Omega} \frac{(\pd_{\tau} u)^2}{ \sqrt{1+|Du|^2}} e^{-\frac{|\xi|^2 +u^2}{2}}\ d\xi < \infty
	\]
and
	\begin{equation}
	\label{eq:dtau-un}
	\lim _{ n \to \infty}  \int_{n-1}^{n+1}\int_{\Omega} \frac{(\pd_{\tau} u)^2}{ \sqrt{1+|Du|^2}} e^{-\frac{|\xi|^2 +u^2}{2}}\ d\xi =0.
	\end{equation}
Define $u_n(\xi, \tau) = u(\xi, n+\tau)$ for $(\xi, \tau) \in \bar \Omega \times [-1, 1]$. For any constant $\rho>R $, the $u_n$'s are  uniformly bounded in $C^{\alpha+2, \alpha/2+1}(\overline{(B_{\rho} \cap \Omega)} \times [-1, 1])$. We can therefore extract a subsequence that converges in $C^{\beta+2, \beta/2+1}(\overline{(B_{\rho} \cap \Omega)} \times [-1, 1])$, $0<\beta<l$. A Cantor diagonal argument gives us a further subsequence $u_{n_i}$ and a function $u_{\infty}$ such that
	$
	u_{n_i} \to u_{\infty}$ as $i \to \infty$ in the norm $C^{\beta+2, \beta/2+1}(\overline{(B_{\rho} \cap \Omega)} \times [-1, 1])$ for every $ \infty >\rho>R$. 
Since every $u_n$ satisfies 
	$
	\pd_{\tau} u_n = \el u_n$ in $ \Omega \times (-1, 1)$ and 
	$ u_n(\xi, \tau) =f$ on $ \pd B_R \times [-1, 1],
	$
 the limit function $u_{\infty}$ inherits the same properties. 
On the one hand, equation \eqref{eq:dtau-un} and the uniform bounds \eqref{eq:u-bound-all} and \eqref{eq:Du-bound-all}    tell us that $\pd_{\tau} u_{n_i} \to 0$ pointwise. On the other hand, we know that $\pd_{\tau} u_{n_i} \to \pd_{\tau} u_{\infty}$ pointwise. Hence, $\pd_{\tau} u_{\infty} =0$ and $u_{\infty}$ satisfies $\el u_{\infty} =0$ in $\Omega$ and $u_{\infty} =f$, $\pd \Omega$. We have therefore proved
 
 \begin{lemma} 
 \label{lem:prelim}
 Let $2> R>\frac{1}{\sqrt 2} $ and $N\geq 5$. Define $\Omega= \{ \xi \in \R^2 \mid |\xi| >R\}$ to be the plane with a hole of radius $R$. There is a $\eps_1 >0$ depending on $R$ and $N$ such that for any $f: [0, 2 \pi] \to \R$ with $\| f \|_{C^4[0, 2 \pi]}=\eps \leq \eps_1$ and satisfying the symmetries $ f(\theta) =- f(- \theta) = f(\pi/N -\theta)$, there exists a smooth solution $u$ to the  Dirichlet problem 
 	\begin{gather}
	\tag{\ref{eq:quasilinear}}
	\el(u)=g^{ij}(Du(\xi))D_{ij}u(\xi) - \xi \cdot D u(\xi) + u(\xi) = 0, \quad \xi \in \Omega,\\
	u =f \textrm{ on } \pd B_R. \tag{\ref{eq:boundary-condition}}
	\end{gather} 
In addition, We can choose the constant $\eps_1=\eps_1(R_0)$ uniformly for all $R$ such that  $2>R\geq R_0>1/\sqrt 2$.
\end{lemma}
%

\section{Uniqueness}

For the sake of clarity let us assume $R>\sqrt 3/2$ in this section. The same reasoning also works in the more general case $R> 1/\sqrt 2$ although the estimates are more involved.

Let $f$ be our boundary condition, and suppose that $\| f \|_{C^4(\pd \Omega)} \leq \eps_1$, where $\eps_1$ is given by Lemma \ref{lem:prelim}.
\begin{theorem}
\label{thm:uniqueness}
  There is an $\eps_2 \leq \eps_1$ so that, if $\| f \|_{C^4} \leq \eps_2$ and 
if $u$ and $\tilde u$ are  two solutions to
	\begin{subequations}
	\label{eq:elliptic-problem}
	\begin{align}
	\label{eq:ssim}
	\el u= g^{ij}(D u) D_{ij} u - \xi \cdot Du +u =0 \text{ in } \Omega \\
	\label{eq:bcond}
	u =f \text{ on } \pd \Omega
	\end{align}
	\end{subequations}
with $u_- \leq u, \tilde u \leq u_+$ then $u \equiv \tilde u$.
\end{theorem}
Before we start the proof, we need some a priori estimates on the solutions $u$ and $\tilde u$.

\subsection{A priori estimates} 
\label{ssec:apriori} 
Because the solutions $u$ and $\tilde u$ above are not dependent on time, they satisfy the parabolic equation $\pd_{\tau} u = \el u$ as well.  Estimates similar to the ones from Lemmas \ref{lem:boundDu} and \ref{lem:D2uaway}  therefore hold in this case also, with some small modifications. 
If $u$ is a solution to the elliptic equation \eqref{eq:ssim},  then the function 
	\[
	v(x,t) = \sqrt{2(1-t)} \; u\bigl(\frac{x}{\sqrt{2(1-t)}}\bigr)
	\]
is a solution to the mean curvature flow on $\mathcal Q_{1}=\{ (x, t) \mid 0<t<1, |x| > \sqrt{2(1-t) R} \}$. The estimate 
\begin{multline} 
		\tag{\ref{eq:ecker-huisken}}
	\sqrt{1+|Dv(x_0, t)|^2} \leq C_1(n)\sup_{B_{\rho}(x_0)} \sqrt{1+|Dv(x, 0)|^2}\\
  	\exp[C_2(n) \rho^{-2} \sup_{[0,T]}(\sup_{B_{\rho}(x_0)\times[0, T]} v -v(x_0, t))^2] 		
	\end{multline}
where $0 \leq t \leq T, B_{\rho}(x_0)$ is a ball in $\Omega$ and $n$ is the dimension of the graph of  $v$, is valid. The last factor of the right hand side can be bounded as in the proof of Lemma \ref{lem:boundDu}; however, we do not have an initial condition independent of $v$ in this case, so right now, our estimates depend on the value of $Dv( \cdot, 0)$ over a bounded set. Taking $|x_0| = \frac{3}{2} \sqrt 2 R$ and $\rho=|x_0|/3$, we obtain
	\[
	|Dv(x_0, t)| \leq C \sup_{2 \sqrt 2 R> |x| > 
	\sqrt 2 R} \sqrt{1+|Dv(x, 0)|^2}, \quad 1>t>0,
	\]
hence, by the change of variable $x =\xi \sqrt{2(1-t)}$, 
 	\[
	|Du(\xi)| \leq C \sup_{2R > |\zeta| > R}  \sqrt{1+|Du(\zeta)|^2}, \quad  |\xi| \geq R
	\]
An argument similar to the one in the  proof of Lemma \ref{lem:D2uaway} gives us bounds on higher derivatives of $u$ away from the boundary
	\begin{align*}
	|D_{\xi\xi}^2 u(\xi)| & \leq  C , \quad |\xi| \geq e (R+1),  \\
|D_{\xi\xi}^3 u(\xi)| &\leq C, \quad |\xi| \geq e (R+1), 
	\end{align*}
with the constant $C$ depending on $\sup_{2R \geq |\zeta| \geq R}  \sqrt{1+|Du(\zeta)|^2}$. Standard elliptic theory on quasilinear equations  (see Theorem 13.4 or 13.7 \cite{gilbarg-trudinger;elliptic-equations} for example)  assures the existence of  constants $C$ and $\alpha>0$ depending only on $f$, $R$ and  $\sup_{2R > |\zeta| > R}  \sqrt{1+|Du(\zeta)|^2}$ such that 
	\[
	\| Du\|_{C^{\alpha}(A_{R, 2R})} \leq C,
	\]
where $A_{R, 2R}$ is the annulus $\{ R<|\xi|<2R\}$. The $C^{\alpha}$ H\"older norm of the coefficient $g^{ij}(Du)$, seen as a function of $\xi$, is uniformly bounded over the whole domain $\Omega$. We have a  bound on $\| D^2_{\xi\xi} u \|_{C^{\alpha}(A_{R, 2R})}$ by elliptic theory; therefore $\| D^2_{\xi\xi} u \|_{C^{\alpha}(\Omega)}$ is also bounded.



\subsubsection{The gradient $Du$ achieves its maximum on the boundary $\pd \Omega$.}
\label{sssec:Du}
Equation \eqref{eq:parabolicDu} from section \ref{sec:parabolic-equation} implies that  $w=|Du|^2$ satisfies
	\begin{equation}
    	\label{eq:lin-for-w0}
    	0 \leq g^{ij}(Du) D_{ij} w + Z_kD_kw - \xi \cdot Dw = :\bar{\lin} w
	\end{equation}
where $ Z_k(x)  = g^{ij}_k(Du)D_{ij}u$ and $g^{ij}_k(p) = \frac{\pd g^{ij}}{\pd p_k}(p)$.
We will derive the following theorem as an application of the maximum principle:

\begin{theorem}
\label{thm:max-gradient-on-R}
Assume $u\in C^2(\Omega) \cap C^0(\bar \Omega)$ is a solution of \eqref{eq:elliptic-problem}, then $|D u|$ achieves its maximum on the circle $r=R$.
\end{theorem}

From the discussion above, $|Z|$ is uniformly bounded and $\bar{\lin}$ is a uniformly elliptic operator on $\Omega$. Moreover, the coefficients of $\bar{\lin}$ are H\"older continuous with  bounded H\"older norm on any compact set of $\Omega$. 
\begin{lemma}
\label{lem:psiexists}
If $
|\nabla h| < b \textrm{  for some constant }b,
$ 
there is a function $\psi(r) >0$ for which $\bar{\lin}(\psi) \leq 0$ and $\psi(r) \to \infty$ as $r \to \infty$. 
\end{lemma}

\begin{proof}
We consider positive increasing functions $\psi(r)$. The operator $\bar{\lin}(\psi)$  is estimated by
	\begin{align*}
	\bar{\lin} (\psi) &= g^{ij}(Du) \Big(\big(\frac{\delta_{ij}}{r} - \frac{\xi_i \xi_j}{r^3} \big) \psi' +\frac{\xi_i \xi_j}{r^2} \psi'' \Big) + \frac{1}{r} \psi' Z \cdot \xi - r \psi' \\
	& \leq g^{ij}(Du) \frac{\xi_i \xi_j}{r^2} \psi '' +\Big( \frac{1}{r} -r + |Z| \Big) \psi'.
	\end{align*}
We have uniform bounds on $|Du|$ and $|Z|$, therefore
	\begin{align*}
	\bar{\lin} (\psi) \leq \Big(g^{ij}(Du) \frac{\xi_i \xi_j}{r^2}\Big) \Big( \psi '' + \psi' ( \frac{1}{r} -r + C ) (1+b^2)\Big).
	\end{align*}
The solution to the differential equation 
	$
	\psi''+\psi'\left(\frac{1}{r}+C-r\right)(1+b^2)=0
	$ is given by 
	\begin{gather*}
	\ln(\psi'(r))-\ln(\psi'(R))=(1+b^2)\left(-\ln(\frac{r}{R})-C(r-R)+\frac{1}{2}(r^2-R^2)\right)\\
	\psi(r)-\psi(R)=\psi'(R) \int_R^r \left(\frac{R}{s}\right)^{1+b^2}e^{(-C(s-R)+\frac{1}{2}(s^2-R^2))(1+b^2)}ds.
	\end{gather*}
For the choice of boundary conditions $\psi(R) = \psi'(R)= \eps>0$, the function $\psi$ and its derivative $\psi'$ are positive for $r>R$. The last thing to verify is that $\psi$ tends to infinity as $r$ approaches infinity. This is true since for large $r$, the dominant term in the expression for $\psi'$ is of order $e^{r^2}$. 
\end{proof}

\begin{proof} [Proof of Theorem \ref{thm:max-gradient-on-R}]
Let us consider the function $w-\eta\psi$, with $\psi$ as in Lemma \ref{lem:psiexists} and $w=|Dh|^2$. This function is negative for large values of $r$. From equation \eqref{eq:lin-for-w0}, we have
	\[
	\bar{\lin}(w-\eta\psi) \geq 0,
	\]
for any positive constant $\eta$. By the maximum principle, $w-\eta \psi$ does not achieve an interior maximum. Since the function $w-\eta \psi$ goes to $-\infty$ as $r$ goes to infinity, it has to achieve its maximum on the boundary $\pd \Omega$, 
	\[
	w(\xi) -\eta \psi(|\xi|)  \leq \max_{\pd \Omega} w - \eta \eps, \quad \xi \in \Omega, \eta>0
	\]
Letting $\eta$ tend to $0$, we obtain
	$
	w(\xi) \leq \max_{\pd \Omega} w$ for $ \xi \in \Omega$.
\end{proof}
Our subsolution and supersolution provide barriers for the function $u$, therefore
	\[
	|Du| \leq C \| f \|_{C^4} 
	\]
with the constant $C$ independent of $f$ and $u$. With this bound, we can run through the arguments in the beginning of this section again and obtain estimates for the second derivative $D^2 u$ that are independent of the solution $u$.

\subsection{Proof of uniqueness}
%
\begin{proof}[ Proof of Theorem \ref{thm:uniqueness}] 
Let $u$ and $\tilde u$ be two solutions as described in the hypotheses. We substract equation \eqref{eq:ssim} for $\tilde u$ from the equation for $u$ and use  the notation $w =u-\tilde u$ to obtain
	\[ 
	g^{ij} (Du) D_{ij} w + [g^{ij}(Du) - g^{ij} (D\tilde u)] D_{ij} \tilde u - \xi \cdot Dw+w =0.
	\]
The Mean Value Theorem implies that
	\[ 
	g^{ij} (Du) D_{ij} w + [g^{ij}_{k}(D\tilde u+ \theta (Du-D\tilde u))] D_{ij} \tilde u D_k w - \xi \cdot Dw+w =0,
	\]
for some $\theta \in (0, 1)$ and where $g^{ij}_k (p) = \frac{\pd g^{ij}}{\pd p_k}$ for $p \in \R^2$. Denoting by $\bar a$ the vector with components $a_k =  [g^{ij}_{k}(D\tilde u+ \theta(Du- D\tilde u))] D_{ij} \tilde u$,  $k=1, 2$, we have
	\begin{equation}
	\label{eq:lin-for-w}
	\bar{\lin}(w)=g^{ij} (Du) D_{ij} w +\bar a \cdot Dw - \xi \cdot Dw+w =0.
	\end{equation}
A computation shows that 
	$
	g^{ij}_k(p) \leq 6 |p| \text{ if } |p| <1 $. From Section \ref{ssec:apriori}, we know that there exists a constant $C$ so that $
	|D_{ij} \tilde u| \leq C \text{ in } \Omega.$
Moreover, the maximum of the first derivative of any solution to \eqref{eq:ssim} is achieved on the boundary $\pd \Omega$ 
hence $ \max_{\Omega} |Du|$ and $\max_{\Omega} |D\tilde u|$ are controlled by $\eps=\| f \|_{C^4}$ and  go to zero as $\eps$ tends to $0$. The same fact is true for $\eta = \sup |\bar a|$ by the definition of $\bar a$.

The maximum principle does not apply immediately to \eqref{eq:lin-for-w} because of the positive coefficient in front of $w$. To circumvent this, we use a positive supersolution $\psi$ to \eqref{eq:lin-for-w} which grows faster than linearly at infinity. The existence of such a supersolution $\psi$ is given below and is the main difficulty in this proof.

Assuming such a positive supersolution $\psi$ exists,  we define $\tilde w = w/\psi$ and look at
	\[
	\tilde \lin (\tilde w) :=\bar{ \lin} (\psi \tilde w) = \bar{\lin} w = 0.
	\]
The coefficient in front of $\tilde w$ in the explicit expression of $\tilde \lin (\tilde w)$ is now $\lin(\psi)$, which is non positive, so the maximum principle applies to $\tilde \lin$. It implies that for any subdomain $\Omega'$ of $\Omega$, 
	\[ 
	\sup_{ \Omega'} \tilde w \leq \sup_{\pd \Omega'} \tilde w.
	\]
From the hypotheses, $w=0$ on $\pd \Omega$ and from the growth of $\psi$ at infinity, $|\tilde w(\xi)|$ tends to $0$ as $|\xi|$ tends to infinity. Choosing the domains $\Omega'$ to be increasingly large annuli, we obtain that $\sup_{\Omega} \tilde w =0$. Since $\tilde w = w/\psi$ and $\psi>0$, $\sup_{\Omega} w =0$, therefore $u \leq \tilde u$ in the entire domain $\Omega$. Switching the roles of $u$ and $\tilde u$ we show that $\tilde u \leq u$ and conclude that  $u \equiv \tilde u$ in $\Omega$.

We are left to prove the existence of the function $\psi$.
\begin{claim} For $R>\sqrt 3/2$ and $\eps_2$ small enough, there exists a  function $\psi $ on $\Omega$ satisfying 
	\[
	\psi >0, \quad \bar{\lin} \psi \leq 0, \quad \text{ and } \lim_{|\xi| \to \infty} \frac{ |\xi|}{\psi (\xi)} =0.
	\]
\end{claim}
	 The supersolution is found among  functions  depending on the radius $r$ only, $\psi = \psi(r)$, that are increasing $\psi'(r)>0$. The operator $\lin$ applied to a radial function is estimated by 
	\begin{align*}
	\bar{\lin} (\psi) &= g^{ij}(Du) \Big[\big(\frac{\delta_{ij}}{r} - \frac{\xi_i \xi_j}{r^3} \big) \psi' +\frac{\xi_i \xi_j}{r^2} \psi'' \Big] + r^{-1} \psi' \bar a \cdot \xi - r \psi'+ \psi \\
		& \leq g^{ij}(Du) \frac{\xi_i \xi_j}{r^2} \psi '' +\Big[ \frac{1}{r} -r + \eta \Big] \psi'+\psi.
	\end{align*}
 Let us consider functions of the form 
	\[
	\psi (r) = r^{\alpha} - b r^{\alpha-2}
	\]
with constants $2>\alpha>1$  to be chosen later  and $b=3/4$. This function is positive on the set $r>R$ provided $b< R^2$. To simplify the notations, we work with the case $b=3/4$ and imposed $R^2>3/4$; however, the following argument is also valid in general for $R^2>b>1/2$ with slightly more subtle estimates. The first and second derivatives are given by
	$
	\psi'=\alpha r^{\alpha-1} -b (\alpha-2) r^{\alpha-3}>0$
and $
	\psi''= \alpha (\alpha-1) r^{\alpha-2} - b (\alpha -2)(\alpha-3) r^{\alpha-4}$. We then have
	\[
	g^{ij}(Du)\frac{\xi_i \xi_j}{r^2} \psi'' \leq \alpha (\alpha-1) r^{\alpha-2} - \frac{1}{1+|Du|^2} b (\alpha-2)(\alpha-3) r^{\alpha-4}.
	\]
By choosing $\eps$ small enough, we can assume that $|Du|^2 \leq 1/8$. A computation gives
	\begin{multline*}
	\bar{\lin} (\psi) \leq (1-\alpha) r^{\alpha} + \eta \alpha r^{\alpha-1} + (\alpha^2 + b\alpha -3b) r^{\alpha-2} \\+ \eta b (2-\alpha) r^{\alpha-3} + b (2-\alpha) (1-\frac{8}{9}(3-\alpha))r^{\alpha-4}
	\end{multline*}
The coefficient of $r^{\alpha-2}$ is negative if $1<\alpha\leq 9/8 <\frac{-3+\sqrt{153}}{8}$ and $b=3/4$. Such a choice of $\alpha$ also guarantees that the coefficient of $r^{\alpha-4}$ is negative. Therefore	\begin{align}
	\bar{\lin} (\psi) &\leq r^{\alpha-4}\Big((1-\alpha) r^4 + \eta \alpha r^3 +  \eta \frac{3}{4} (2-\alpha) r \Big) \notag\\
		&\leq r^{\alpha-3}\Big((1-\alpha) r + 2\eta  \Big) \label{eq:linpsi}
	\end{align}
since $ \frac{3}{4} r \leq r^3$. The right hand side of \eqref{eq:linpsi} is negative if we choose $\eps$ small enough and $\alpha$ so that 
	\[
	1+\frac{4 \eta}{\sqrt 3} < \alpha <\frac{9}{8}. \qedhere
	\]
\end{proof}

\section{Remarks}
Let  $\|f \|_{C^4} \leq \eps$ and let $u$ be the unique solution to \eqref{eq:elliptic-problem} such that $u_- \leq u\leq u_+$. 

\begin{remark}
Because of uniqueness, the function $u$ has to satisfy the symmetries 
	\begin{equation}
	u(r, \theta)=-u(r, -\theta) = u(r, \pi/N -\theta) \text{ for } r >R, \theta \in [0, 2\pi)\tag{\ref{eq:symmetries-u}}.
	\end{equation}
Taking $\eps_0= \eps_2$, Theorem \ref{thm:main-1} from the introduction is then a corollary of Lemma \ref{lem:prelim} and Theorem \ref{thm:uniqueness}.
\end{remark}
\begin{remark}
From the explicit formulas in Lemma \ref{lem:subsuper}, we have
	\[
	u_- \sim u_1 - \eps^3 v_C \leq u \leq u_+ \sim u_1 + \eps^3 v_C,
	\]
where  the constant $C$ and the function $v_C$ do not depend on $\eps$. Recall that $u_1$ is the constructed solution to the linear problem $\lin u_1 =0$, $u_1|_{\pd \Omega} =f$. This implies that  $D_{r} u = D_r u_1 + O(\eps^3)$ on the boundary circle $r=R$, in other words, the radial derivative of $u$ on the circle $r=R$ is the radial derivative of the solution  to the linear problem with an error of order $\eps^3$. 
\end{remark}
\begin{remark}
Considering equation \eqref{eq:D2u-first-step} with  $|\xi| = e^{\tau} (R+1)$ gives us
	\[
	|D^2_{\xi\xi} u(\xi)| \leq \frac{C}{|\xi|}, \quad |\xi| \geq  e (R+1).
	\]
Since $|Du|$ is uniformly bounded in $\Omega$, the above estimate implies a bound on the mean curvature of the graph, and therefore on $X \cdot \nu$:
	\[
	| X \cdot \nu| =| H | \leq \frac{C}{|\xi|}, \quad |\xi| \geq  e (R+1).
	\]
The graph of the solution $u$ is therefore asymptotic to a cone at infinity. 
\end{remark}

\bibliographystyle{siam}
\bibliography{thesis}

\end{document}